\documentclass[reqno, 11pt]{amsart}
\usepackage[utf8]{inputenc}
\usepackage{amssymb,amsmath,amsfonts,amsthm,calrsfs,mathtools}
\usepackage[all]{xy}
\usepackage{color}
\usepackage{float}
\usepackage{orcidlink}
\usepackage{subcaption}
\usepackage[english]{babel}
\usepackage[T1]{fontenc}
\usepackage{graphicx}
\usepackage{multicol}
\usepackage{latexsym}
\usepackage{enumitem}
\usepackage{tabularx}
\usepackage{mdframed}
\usepackage{booktabs}
\usepackage{xcolor}
\usepackage{tikz}
\usepackage{pgfplots}
\pgfplotsset{compat=1.18}
\usetikzlibrary{arrows.meta,decorations.markings,decorations.pathreplacing,calc,patterns,hobby,positioning}

\definecolor{teal}{RGB}{15,110,86}
\definecolor{coral}{RGB}{153,60,29}
\definecolor{purple}{RGB}{83,74,183}
\definecolor{amber}{RGB}{186,117,23}
\definecolor{gray1}{RGB}{68,68,65}
\usepackage[a4paper,top=3cm,bottom=2cm,left=3cm,right=3cm,marginparwidth=1.75cm]{geometry}
\usepackage[colorinlistoftodos]{todonotes}

\tikzset{
  arr/.style={-{Stealth[length=5pt,width=4pt]},thick},
  arr_thin/.style={-{Stealth[length=4pt,width=3pt]},thin},
}

\theoremstyle{plain}
\newtheorem{theorem}{Theorem}
\newtheorem{lemma}{Lemma}
\newtheorem{proposition}{Proposition}
\newtheorem{corollary}{Corollary}

\newtheorem{maintheorem}{Theorem}

\theoremstyle{definition}

\newtheorem{remark}{Remark}

\DeclarePairedDelimiter{\floor}{\lfloor}{\rfloor}
\DeclareMathOperator{\sgn}{sgn}

\title{
Limit Cycles in a Discontinuous Generalized Liénard Systems with Bivariate Perturbation}
\author{Carlos F.\ \'{A}lvarez \orcidlink{0000-0001-5717-6531}}
\author{Alexander Sierra-Ortíz \orcidlink{0009-0000-5961-7062}}
\address{Departamento de Matem\'{a}ticas, Universidad del Atl\'{a}ntico, Cra~30 \#~8-49, 081001, Puerto Colombia, Colombia}
\email[(Carlos F.\ \'{A}lvarez)]{cfalvarez@mail.uniatlantico.edu.co}
\email[(Alexander Sierra-Ortiz)]{ajosesierra@mail.uniatlantico.edu.co}

\date{\today}

\begin{document}
 
\begin{abstract}
We study the number of limit cycles of the piecewise smooth differential system \[ \dot{x}=y, \quad \dot{y}=-x-\varepsilon\bigl(f(x)y+\operatorname{sgn}(x)g(x,y)\bigr),\] 
where $f(x)$ is a polynomial of degree $n\geq 1$, $g(x,y)$ is a bivariate polynomial of degree $m\geq 2$, and $\varepsilon$ is a sufficiently small parameter. Using the first-order averaging theorem for discontinuous systems, we obtain that the number $H_{m,n}=\lfloor n/2\rfloor+\lfloor m/2\rfloor$ is a lower bound for the number of limit cycles bifurcating from the linear center $\dot{x}=y, \ \dot{y}=-x$. The bound is sharp, and we provide an explicit example that attains $H_{m,n}$ limit cycles.
\end{abstract}

\maketitle
\markright{}

\section{Introduction}

The study of oscillations is central to nonlinear dynamics: many physical systems, from mechanical oscillators to electrical circuits, settle into a repeating motion that is stable under small perturbations. Mathematically, such a motion is a \emph{limit cycle}, an isolated periodic orbit of a differential equation, meaning that nearby trajectories either converge to it or move away from it, but no other periodic orbit lies arbitrarily close. A classical model exhibiting this behaviour is the \emph{generalized Liénard equation}
\begin{equation}
\ddot{x} + f(x)\dot{x} + g(x) = 0, \label{eq:lienard_smooth}
\end{equation}
introduced by Alfred Liénard in 1928~\cite{LienardOriginal1928}. Explicit solutions of \eqref{eq:lienard_smooth} are known only for special choices of $f$ and $g$~\cite{Feng2004}, so a qualitative approach to its dynamics is generally required. When $f$ and $g$ are polynomials, counting the limit cycles of~\eqref{eq:lienard_smooth} is part of Hilbert's 16th problem, which asks for a bound on the number of limit cycles of planar polynomial systems in terms of their degree~\cite{HilbertProblem, Ilyashenko2002}. Lins, de~Melo and Pugh~\cite{LinsDelMeloPugh1977} conjectured in 1977 that the classical equation $\ddot{x}+f(x)\dot{x}+x=0$, with $\deg f=n$, has at most $\lfloor n/2\rfloor$ limit cycles. This was proved for $n\leq 3$ but disproved for $n\geq 6$ by De~Maesschalck and Dumortier~\cite{DeMaesschalckDumortier2011}, who exhibited $\lfloor(n-1)/2\rfloor+2$ limit cycles in that range.

In addition to the smooth case, discontinuous Liénard equations arise naturally in mechanicscontrol theory, and electrical engineering, where the right-hand side switches abruptly across a surface~\cite{FilippovConvention}. The general case of Hilbert's 16th problem for such piecewise smooth systems remains open, but several particular cases have been solved: Llibre, Ram\'irez and Sadovskaia~\cite{LlibreRamirezSadovskaia2010} solved the problem for algebraic limit cycles, and~\cite{LlibreNovaesTeixeira2015, LlibreZhang2019} treated further particular cases and discontinuous piecewise linear differential systems. Martins and Mereu~\cite{MartinsMereu2014} studied
\begin{equation}
  \dot{x}=y,\qquad
  \dot{y}=-x-\varepsilon\bigl(f(x)\,y+\sgn(y)(k_1 x+k_2)\bigr),
\end{equation}
with $\deg f=n$ and $k_1,k_2\in\mathbb{R}$, and proved via first-order averaging that $\lfloor n/2\rfloor+1$ is the exact maximum number of limit cycles bifurcating from the linear center. Diab, Guirao and Vera~\cite{DiabGuiraoVera2022} later made the \emph{smooth} term bivariate, $f(x,y)$, while keeping the discontinuous term $\phi_w(y)(k_1x+k_2)$ linear in $x$ as in~\cite{MartinsMereu2014}, and obtained the same bound; their averaging integral still annihilates every monomial with $i$ or $j$ odd, so the bivariate extension does not change which terms contribute. They also work with a piecewise-linear regularization $\phi_w(y)$ of $\sgn(y)$, recovering the discontinuous case only as $w\to0$, whereas the present paper applies the first-order averaging theorem for discontinuous systems directly, following~\cite{DeAbreu2024,MartinsMereu2014}. Abreu and Martins~\cite{DeAbreu2024} replaced the linear term $k_1x+k_2$ by a polynomial $g(x)$ of arbitrary degree $m\geq1$:

\begin{theorem}[Abreu--Martins~\cite{DeAbreu2024}]\label{thm:AbMar}
Let $f(x)$ and $g(x)$ be real polynomials of degrees $n\geq 1$ and $m\geq 1$ respectively. For $|\varepsilon|$ sufficiently small, the number $h_{m,n}=\lfloor n/2\rfloor+\lfloor m/2\rfloor+1$ is a lower bound for the maximum number of limit cycles of
\begin{equation}
  \dot{x}=y,\qquad
  \dot{y}=-x-\varepsilon\bigl(f(x)\,y+\sgn(y)\,g(x)\bigr)
\end{equation}
bifurcating from the periodic orbits of the linear center. Moreover, this bound is achieved for suitable choices of $f$ and $g$.
\end{theorem}

In all of the above works, the discontinuity occurs on the horizontal axis $\Sigma=\{y=0\}$. This paper asks what changes when the discontinuity is placed instead on the vertical axis $\Sigma=\{x=0\}$, using the first-order averaging theorem for discontinuous systems of Llibre, Mereu and Novaes~\cite{LlibreMereuNovaes2015}. The switching points then shift from $\theta=0,\pi$ to $\theta=\pi/2,3\pi/2$, and this shift has a direct consequence: for \emph{any} univariate perturbation $g(x)$, the corresponding contribution to the first-order averaged function vanishes identically, regardless of the degree of $g$. Recovering a nontrivial averaged function therefore requires $g$ to depend on $y$ and $x$. We study the system
\begin{equation}
\dot{x}=y, \quad \dot{y}=-x-\varepsilon\bigl(f(x)y+\text{sgn}(x)g(x,y)\bigr), \label{eq:main_sys}
\end{equation}
where $f$ has degree $n\geq1$ and $g$ is a bivariate polynomial of degree $m\geq2$; the degree hypothesis $m\geq2$ is sharp, since it is the least degree at which a monomial $x^iy^j$ can have both indices odd, and a symmetry argument shows that for $m=1$ none of the monomials $1, x,$ and $y$
occurring in g contributes to $F_0$, collapsing the system to the smooth bound $H_{1,n}=\lfloor n/2\rfloor$; this contrasts with the horizontal case (Theorem~\ref{thm:AbMar}), where $m=1$ already contributes. Our main result is the following.

\begin{maintheorem}[Main Result]\label{thm:main}
Let $f(x)=\sum_{i=0}^n a_ix^i$ be a real polynomial of degree $n\geq 1$ and let $g(x,y)=\sum_{i+j\leq m}b_{ij}x^iy^j$ be a real bivariate polynomial of degree $m\geq 2$. Consider system~\eqref{eq:main_sys}. For $|\varepsilon|$ sufficiently small, the number $H_{m,n}=\Bigl\lfloor\frac{n}{2}\Bigr\rfloor+\Bigl\lfloor\frac{m}{2}\Bigr\rfloor$ is a lower bound for the maximum number of limit cycles of~\eqref{eq:main_sys} bifurcating from the periodic orbits of the linear center $\dot{x}=y$, $\dot{y}=-x$. Moreover, there exist polynomials $f$ and $g$ such that exactly $H_{m,n}$ limit cycles are achieved.
\end{maintheorem}

The bound $H_{m,n}$ is exactly one less than the Abreu--Martins bound $h_{m,n}$; as we will show in Remark~\ref{rem:no_r0}, this is because the vertical switching mechanism admits no contribution of order $r^0$, so the lowest power of $r$ appearing from $g$ is $r^2$ rather than $r^0$. The proof bounds the number of positive roots of $F_0$ via Descartes' rule of signs~\cite{Grabiner1999} and shows the bound is attained using an elementary root-realization argument.

\section{Preliminaries}

We use the framework of Filippov for piecewise smooth systems, together with the first-order averaging theorem for discontinuous systems of Llibre, Mereu and Novaes~\cite{LlibreMereuNovaes2015}.

Let $U \subset \mathbb{R}^2$ be an open neighbourhood of $0$. Since every embedded hypersurface is, locally, the inverse image of a regular value, we let $\Sigma = h^{-1}(0) \cap U$, where $h$ is the \emph{germ} at $0$ of a $C^r$ function, $r \geq 1$, having $0$ as a regular value and two such functions are identified whenever they agree on a common neighbourhood of $0$. 

Note that the hypersurface splits $U$ into the following open sets: 
\[\Sigma^+=\{p\in U:h(p)>0\} \quad \Sigma^-=\{p\in U:h(p)<0\}.\]

A piecewise smooth vector field $Z=(X,Y)$ is given by a smooth $X$ on $\overline{\Sigma^+}$ and a smooth $Y$ on $\overline{\Sigma^-}$. For $p\in\Sigma$, the Lie derivatives $Xh(p)=X(p)\cdot\nabla h(p)$ and $Yh(p)=Y(p)\cdot\nabla h(p)$ measure whether the flows of $X$ and $Y$ point toward $\Sigma^+$ or $\Sigma^-$ at $p$, and together they determine the local behaviour of the piecewise vector field near $p$. The point $p$ belongs to the \emph{crossing region}
\[
  \Sigma_c=\{p\in\Sigma:Xh(p)\cdot Yh(p)>0\},
\]
the \emph{sliding region} $\Sigma_s=\{p\in\Sigma:Xh(p)<0,\,Yh(p)>0\}$, or the \emph{escaping region} $\Sigma_e=\{p\in\Sigma:Xh(p)>0,\,Yh(p)<0\}$. If $p\in\Sigma_c$, the trajectory through $p$ is obtained by matching the trajectories of $X$ and $Y$; if $p\in\Sigma_s\cup\Sigma_e$, one instead resorts to the Filippov convention. We assume tangency points ($Xh(p)=0$ or $Yh(p)=0$) are isolated in $\Sigma$. Figure~\ref{fig:crossing-sliding-escaping} illustrates the three cases.

\begin{figure}[h]
\centering
\begin{subfigure}[b]{0.3\textwidth}
\centering
\begin{tikzpicture}[scale=1.0, >=Stealth, line width=0.8pt]
  \draw[red, thick] plot[smooth] coordinates
    {(0.3,3.0) (0.8,2.75) (1.3,2.05) (1.9,1.6) (2.4,1.43) (3.0,1.38) (3.5,1.38) (4.0,1.33) (4.55,1.13) (5.1,0.7) (5.5,0.2)};
  \node[red] at (5.85,0.45) {$\Sigma$};
  \foreach \x/\y in {2.4/1.42, 3.05/1.38, 3.7/1.34, 4.4/1.18} {
    \draw[->] ($(\x,\y)+(0.29,-0.79)$) -- (\x,\y);
    \draw[->] (\x,\y) -- ($(\x,\y)+(0.58,0.60)$);
  }
\end{tikzpicture}

\caption{Crossing region}
\end{subfigure}
\hfill
\begin{subfigure}[b]{0.3\textwidth}
\centering
\begin{tikzpicture}[scale=1.0, >=Stealth, line width=0.8pt]
  \draw[red, thick] plot[smooth] coordinates
    {(0.3,3.0) (0.8,2.75) (1.3,2.05) (1.9,1.6) (2.4,1.43) (3.0,1.38) (3.5,1.38) (4.0,1.33) (4.55,1.13) (5.1,0.7) (5.5,0.2)};
  \node[red] at (5.85,0.45) {$\Sigma$};
  \foreach \x/\y in {2.4/1.42, 3.05/1.38, 3.7/1.34, 4.4/1.18} {
    \draw[->] ($(\x,\y)+(-0.58,-0.60)$) -- (\x,\y);
    \draw[->] ($(\x,\y)+(-0.29,0.79)$) -- (\x,\y);
  }
\end{tikzpicture}
\caption{Sliding region}
\end{subfigure}
\hfill
\begin{subfigure}[b]{0.3\textwidth}
\centering
\begin{tikzpicture}[scale=1.0, >=Stealth, line width=0.8pt]
  \draw[red, thick] plot[smooth] coordinates
    {(0.3,3.0) (0.8,2.75) (1.3,2.05) (1.9,1.6) (2.4,1.43) (3.0,1.38) (3.5,1.38) (4.0,1.33) (4.55,1.13) (5.1,0.7) (5.5,0.2)};
  \node[red] at (5.85,0.45) {$\Sigma$};
  \foreach \x/\y in {2.4/1.42, 3.05/1.38, 3.7/1.34, 4.4/1.18} {
    \draw[->] (\x,\y) -- ($(\x,\y)+(0.29,-0.79)$);
    \draw[->] (\x,\y) -- ($(\x,\y)+(0.58,0.60)$);
  }
\end{tikzpicture}
\caption{Escaping region}
\end{subfigure}

\caption{Crossing, sliding, and escaping regions.}
\label{fig:crossing-sliding-escaping}
\end{figure}
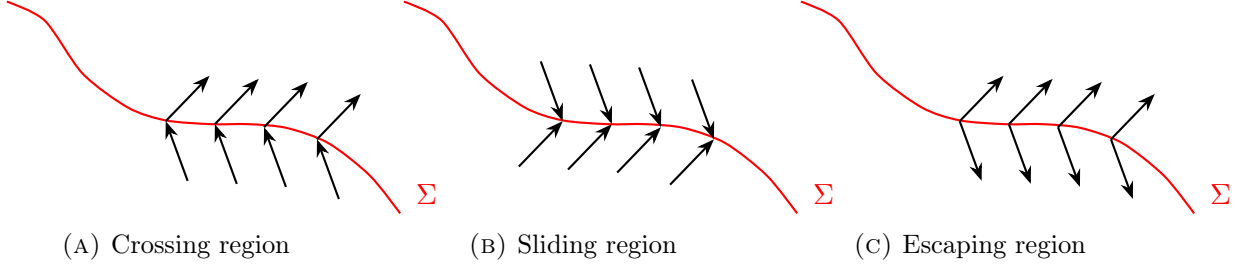

\begin{lemma}\label{lem:crossing}
For system~\eqref{eq:main_sys}, $\Sigma\setminus\{0\}\subset\Sigma_c$.
\end{lemma}
\begin{proof}
Here $h(x,y)=x$, so $\nabla h=(1,0)$ and, at $p=(0,y_0)\in\Sigma$, $Xh(p)=Yh(p)=y_0$. Hence $Xh(p)\cdot Yh(p)=y_0^2>0$ for $y_0\neq0$, so $p\in\Sigma_c$.
\end{proof}

By Lemma~\ref{lem:crossing}, the Filippov convention on $\Sigma_s\cup\Sigma_e$ is not needed for system~\eqref{eq:main_sys}. The $T$-periodic solutions produced by the averaging theorem below are, in the usual sense, limit cycles of \eqref{eq:main_sys} bifurcating from the periodic orbits of the linear center as $\varepsilon$ moves away from 0.

Consider a standard form system
\begin{equation}\label{eq:standard}
  \dot{x}(t)=\varepsilon F(t,x)+\varepsilon^2 R(t,x,\varepsilon),
\end{equation}
where
\[  F(t,x)=\begin{cases}F_1(t,x),&x\in S_1,\\ F_2(t,x),&x\in S_2,\end{cases}
  \qquad
  R(t,x,\varepsilon)=\begin{cases}R_1(t,x,\varepsilon),&x\in S_1,\\
  R_2(t,x,\varepsilon),&x\in S_2,\end{cases}\]
$D\subset\mathbb{R}$ is an open bounded set, $S_1=D\cap\overline{\Sigma^+}$ and $S_2=D\cap\overline{\Sigma^-}$, and $F_{1,2},R_{1,2}$ are $T$-periodic in $t$. The \emph{averaged function} is
\begin{equation}
  F_0(z)=\frac{1}{T}\int_0^T F(t,z)\,dt.
\end{equation}

\begin{theorem}[{\cite[Theorem~A]{LlibreMereuNovaes2015}}]\label{thm:averaging}
Suppose the following hold for system~\eqref{eq:standard}:
\begin{enumerate}[label=\textup{(H\arabic*)}]
\item\label{H1} There exists an open bounded set $C\subset D$ such that for each $z\in C$ the curve $\{(t,z):t\in\mathbb{R}/T\mathbb{Z}\}$ meets $\Sigma$ transversally and only at generic crossing points.
\item\label{H2} For $j=1,2$, the functions $F_j$ and $R_j$ are locally Lipschitz in $x$ and $T$-periodic in $t$; the boundaries of $S_j$ are piecewise $C^1$ embedded hypersurfaces.
\item\label{H3} For $a\in C$ with $F_0(a)=0$ there is a neighbourhood $U\subset C$ of $a$ such that $F_0(z)\neq 0$ for $z\in U\setminus\{a\}$ and $d_B(F_0,U,0)\neq 0$, where $d_B$ denotes the Brouwer degree~\cite{BuicaLlibre2004}.
\end{enumerate}
Then for $|\varepsilon|>0$ sufficiently small there exists a $T$-periodic solution $x(t,\varepsilon)$ of~\eqref{eq:standard} with $x(0,\varepsilon)\to a$ as $\varepsilon\to 0$.
\end{theorem}

By Theorem~\ref{thm:averaging} and Lemma~\ref{lem:crossing}, hypothesis~\ref{H1} holds automatically for system~\eqref{eq:main_sys}; finding limit cycles bifurcating from the linear center reduces to locating the positive zeros of $F_0$, computed in Section~\ref{sec:proof}. We record two tools used there: Descartes' rule of signs, to bound the number of positive roots of $F_0$, and an elementary lemma based on Rolle's theorem, to show this bound is attained with simple roots.

\begin{theorem}[Descartes's Rule of Signs~\cite{Grabiner1999}]\label{thm:descartes}
Let $p(x)=c_{i_1}x^{i_1}+\cdots+c_{i_r}x^{i_r}$ with $0\leq i_1<\cdots<i_r$ and all $c_{i_k}\neq 0$. The number of positive real roots of $p$ (counted with multiplicity) does not exceed the number of sign changes in the sequence $c_{i_1},\ldots,c_{i_r}$. Moreover, by choosing the coefficients appropriately one can achieve exactly $r-1$ positive real roots.
\end{theorem}

The following bound depends only on the number of monomials present, not on their signs; it is implicit in Theorem~\ref{thm:descartes}, but we prove it directly since it is the form used with Proposition~\ref{prop:achieve} in Section~\ref{sec:proof}.

\begin{lemma}[Sparse root bound]\label{lem:sparse}
Let $0\leq s_1<s_2<\cdots<s_t$ be integers and let $c_1,\ldots,c_t\in\mathbb{R}$, not all zero. Then
\[  P(r)=\sum_{j=1}^t c_j\,r^{s_j}\] has at most $t-1$ positive real roots in $(0,\infty)$, counted with multiplicity.
\end{lemma}

\begin{proof}
It suffices to treat the case in which every $c_j\neq 0$: if some coefficients vanish, $P$ is a sparse polynomial of the same type with fewer terms, say $t'<t$, for which the bound $t'-1<t-1$ already gives the (a fortiori stronger) conclusion. So assume $c_1,\ldots,c_t$ are all nonzero, and argue by induction on $t\geq 1$.

\emph{Base case $t=1$.} Here $P(r)=c_1r^{s_1}$ with $c_1\neq 0$, so $P(r)\neq 0$ for every $r>0$: $P$ has $0=t-1$ positive roots.

\emph{Inductive step.} Let $t\geq 2$ and assume the statement holds for sparse polynomials with $t-1$ nonzero terms. Since $r^{s_1}>0$ for $r>0$, the function
\[  Q(r):=r^{-s_1}P(r)=c_1+\sum_{j=2}^t c_j\,r^{e_j},\qquad e_j:=s_j-s_1>0,\]
has exactly the same positive roots as $P$, with the same multiplicities.

We use the following standard consequence of Rolle's theorem. Let $\rho_1<\cdots<\rho_\ell$ be the distinct positive roots of $Q$, with multiplicities $\mu_1,\ldots,\mu_\ell$, so that $Q$ has $k:=\mu_1+\cdots+\mu_\ell$ positive roots counted with multiplicity. If $\mu_i\geq 2$, then $\rho_i$ is a root of $Q'$ of multiplicity $\mu_i-1$; this accounts for $\sum_{i=1}^\ell(\mu_i-1)=k-\ell$ roots of $Q'$ (with multiplicity). In addition, Rolle's theorem furnishes, for each pair of consecutive roots $\rho_i<\rho_{i+1}$, a root of $Q'$ in the open interval $(\rho_i,\rho_{i+1})$; these $\ell-1$ intervals are pairwise disjoint and disjoint from $\{\rho_1,\ldots,\rho_\ell\}$, so they contribute $\ell-1$ further, distinct roots of $Q'$. In total, $Q'$ has at least
\[  (k-\ell)+(\ell-1)=k-1\]
positive roots counted with multiplicity.

On the other hand,
\[  Q'(r)=\sum_{j=2}^t c_j e_j\,r^{e_j-1}\] is itself a sparse polynomial: it has exactly $t-1$ nonzero coefficients $c_je_j$ (nonzero since $c_j\neq 0$ and $e_j\geq 1$), with pairwise distinct nonnegative integer exponents $e_2-1<\cdots<e_t-1$. By the induction hypothesis, $Q'$ has at most $t-2$ positive roots counted with multiplicity.

Combining the two bounds, $k-1\leq t-2$, i.e.\ $k\leq t-1$. Thus $Q$, and hence $P$, has at most $t-1$ positive roots counted with multiplicity, completing the induction.
\end{proof}

\begin{proposition}\label{prop:achieve}
Let $0\leq s_1<\cdots<s_t$ be integers and let $r_1<\cdots<r_{t-1}$ be any $t-1$ pairwise distinct positive reals. Then there exist reals $c_1,\ldots,c_t$, not all zero, such that $P(r):=\sum_{j=1}^tc_jr^{s_j}$ vanishes at $r_1,\ldots,r_{t-1}$. Moreover, for \emph{any} such choice of coefficients, $r_1,\ldots,r_{t-1}$ are precisely the positive roots of $P$, and each of them is simple.
\end{proposition}
\begin{proof}
The conditions $P(r_i)=0$ for $i=1,\ldots,t-1$ form a homogeneous linear system of $t-1$ equations in the $t$ unknowns $c_1,\ldots,c_t$; its solution space has dimension at least $t-(t-1)=1$, so a nonzero solution $(c_1,\ldots,c_t)$ exists.

Fix any such nonzero solution. Then $P$ vanishes at the $t-1$ pairwise distinct points $r_1,\ldots,r_{t-1}$, so $P$ has at least $t-1$ positive roots counted with multiplicity. By Lemma~\ref{lem:sparse}, $P$ has at most $t-1$ positive roots counted with multiplicity. The two bounds force equality: $P$ has exactly $t-1$ positive roots counted with multiplicity, and since $r_1,\ldots,r_{t-1}$ already account for $t-1$ of them (each with multiplicity at least one), each must occur with multiplicity exactly one, and $P$ has no positive root other than $r_1,\ldots,r_{t-1}$.
\end{proof}

\section{The parity obstruction for univariate perturbations}\label{sec:obstruction}

Before proving the main theorem we characterize which monomials of $g(x,y)$ contribute to the averaged function through the discontinuous term, and show as a consequence that no perturbation depending on $x$ alone can contribute, regardless of its degree.

In polar coordinates $x=r\cos\theta$, $y=r\sin\theta$, $r>0$, one has $\sgn(x)=\sgn(r\cos\theta)=\sgn(\cos\theta)$. The contribution of a monomial $x^iy^j$ to the averaged function (see Section~\ref{sec:proof}) is governed by
\begin{equation*}
M_{i,j}=\int_0^{2\pi}\sgn(\cos\theta)\,\sin\theta\,\cos^i\theta\,\sin^j\theta \,d\theta,\qquad i,j\geq 0.
\end{equation*}

\begin{lemma}\label{lem:Mij}
For $i,j\geq 0$. 
\begin{equation}
M_{i,j}=
\begin{cases}
0 & \text{if } i \text{ is even or } j \text{ is even},\\[6pt]
\displaystyle
4\int_0^{\pi/2}\sin^{j+1}\theta\cos^i\theta\,d\theta > 0
& \text{if both } i \text{ and } j \text{ are odd.}
\end{cases}
\end{equation}
\end{lemma}

\begin{proof}
Write $h_{i,j}(\theta)=\sgn(\cos\theta)\sin\theta\cos^i\theta\sin^j\theta =\sgn(\cos\theta)\sin^{j+1}\theta\cos^i\theta$.

First, we observe that, under the shift $\theta\mapsto\theta+\pi$:
\begin{align*}
  h_{i,j}(\theta+\pi)
  &=\sgn(\cos(\theta+\pi))\sin^{j+1}(\theta+\pi)\cos^i(\theta+\pi)\\
  &=\bigl(-\sgn(\cos\theta)\bigr)(-1)^{j+1}\sin^{j+1}\theta\cdot(-1)^i\cos^i\theta\\
  &=(-1)^{i+j+1+1}\,h_{i,j}(\theta)
   =(-1)^{i+j}\,h_{i,j}(\theta).
\end{align*}
Therefore
\begin{itemize}
\item If $i+j$ is \emph{odd}: $h_{i,j}(\theta+\pi)=-h_{i,j}(\theta)$, so
$h_{i,j}$ is $\pi$-antiperiodic.  Then
$\int_\pi^{2\pi}h_{i,j}\,d\theta=-\int_0^\pi h_{i,j}\,d\theta$,
 hence $M_{i,j}=\int_0^{2\pi}h_{i,j}\,d\theta=0$.
\item If $i+j$ is \emph{even}: $h_{i,j}(\theta+\pi)=h_{i,j}(\theta)$, so $h_{i,j}$ is $\pi$-periodic, and $M_{i,j}=2\int_0^\pi h_{i,j}(\theta)\,d\theta$.
\end{itemize}
Since $i+j$ odd if and only if $i$ and $j$ have different parities, we have established that \emph{if $i$ and $j$ have different parities, then $M_{i,j}=0$}.

\medskip
Now, we study the case $i$ and $j$ both even. Suppose $i=2p$, $j=2q$, so $i+j$ is even.  Now apply the substitution $\theta\mapsto -\theta$ to the integrand $h_{2p,2q}(\theta)$:
\begin{align*}
  h_{2p,2q}(-\theta)
  &=\sgn(\cos(-\theta))\sin^{2q+1}(-\theta)\cos^{2p}(-\theta)\\  &=\sgn(\cos\theta)\cdot(-1)^{2q+1}\sin^{2q+1}\theta\cdot\cos^{2p}\theta
   =-h_{2p,2q}(\theta).
\end{align*}
Hence $h_{2p,2q}$ is an odd function of $\theta$, and
$M_{2p,2q}=\int_{-\pi}^{\pi}h_{2p,2q}(\theta)\,d\theta=0$.

Finally, we study the case $i=2p+1$ and $j=2q+1$ both odd. Here $i+j=2(p+q+1)$ is even, so by Step~1,
$M_{i,j}=2\int_0^\pi h_{i,j}(\theta)\,d\theta$.
Apply $\theta\mapsto\pi-\theta$ to $\int_{\pi/2}^{\pi}h_{i,j}\,d\theta$:
\begin{align*}
\int_{\pi/2}^{\pi}h_{2p+1,2q+1}(\theta)\,d\theta
&=\int_0^{\pi/2}h_{2p+1,2q+1}(\pi-\phi)\,d\phi.
\end{align*}
Now $\cos(\pi-\phi)=-\cos\phi$ and $\sin(\pi-\phi)=\sin\phi$, so
\begin{align*}
h_{2p+1,2q+1}(\pi-\phi)
&=\sgn(\cos(\pi-\phi))\sin^{2q+2}(\pi-\phi)\cos^{2p+1}(\pi-\phi)\\
&=\bigl(-\sgn(\cos\phi)\bigr)\sin^{2q+2}\phi\cdot(-\cos\phi)^{2p+1}\\
&=(-1)\cdot\sin^{2q+2}\phi\cdot(-1)^{2p+1}\cos^{2p+1}\phi\cdot\sgn(\cos\phi)\\
&=(-1)(-1)^{2p+1}\,h_{2p+1,2q+1}(\phi)
=h_{2p+1,2q+1}(\phi).
\end{align*}
Therefore $\int_{\pi/2}^{\pi}h_{i,j}=\int_0^{\pi/2}h_{i,j}$, and
$$M_{2p+1,2q+1}=2\int_0^\pi h_{2p+1,2q+1}\,d\theta  
=4\int_0^{\pi/2}\sgn(\cos\theta)\sin^{2q+2}\theta\cos^{2p+1}\theta\,d\theta.$$
Since $\cos\theta>0$ on $(0,\pi/2)$, we have $\sgn(\cos\theta)=1$ there, so
\[
M_{2p+1,2q+1}=4\int_0^{\pi/2}\sin^{2q+2}\theta\cos^{2p+1}\theta\,d\theta>0,\]
the integral being strictly positive because the integrand is positive on
$(0,\pi/2)$.
\end{proof}

\begin{corollary}\label{cor:univariate}
If $g=g(x)$ depends only on $x$, its contribution to the first-order averaged function vanishes identically, regardless of the degree of $g$.
\end{corollary}

\begin{proof}
Writing $g(x)=\sum_j b_j x^j$, each monomial $x^j=x^jy^0$ has second index equal to $0$, which is even. By Lemma~\ref{lem:Mij}, $M_{j,0}=0$ for every $j\geq0$, so every term of $g(x)$ contributes zero.
\end{proof}

The obstruction in Corollary~\ref{cor:univariate} shows that the perturbation must involve $y$ to break the parity; the natural choice is a bivariate $g(x,y)$, for which only monomials with both indices odd survive, by Lemma~\ref{lem:Mij}.

\begin{remark}
For $i=2p+1$ and $j=2q+1$. The integral $M_{2p+1,2q+1}$ can be evaluated using the standard formulas for powers of trigonometric functions, for instance, formula 3.621(5) in \cite{MR2360010} provides that 
\begin{equation}
M_{2p+1,2q+1}
=4\int_0^{\pi/2}\sin^{2q+2}\theta\cos^{2p+1}\theta\,d\theta =\frac{2\Gamma\left(\frac{2q+3}{2}\right)\Gamma(p+1)}{\Gamma\left(\frac{2(p+q)+5}{2}\right)}
\end{equation}
In particular, $M_{2p+1,2q+1}>0$ forall $p,q\geq 0.$
\end{remark}
 
\section{Proof of Theorem~\texorpdfstring{\ref{thm:main}}{A}}\label{sec:proof}
We consider polar coordinates $x=r\cos\theta$, $y=r\sin\theta$, $r>0$, and let
\begin{equation}
a(r,\theta):=f(r\cos\theta)\cdot r\sin\theta+\sgn(\cos\theta)\,g(r\cos\theta,r\sin\theta).
\end{equation}
From~\eqref{eq:main_sys}, since $\dot x=y$ and
$\dot y=-x-\varepsilon(f(x)y+\sgn(x)g(x,y))$, we have $x\dot x+y\dot y=-\varepsilon\,y\bigl(f(x)y+\sgn(x)g(x,y)\bigr)$ and $x\dot y-y\dot x=-r^2-\varepsilon\,x\bigl(f(x)y+\sgn(x)g(x,y)\bigr)$, and therefore, \emph{exactly} in $\varepsilon$,
\[  \dot r=\frac{x\dot x+y\dot y}{r}=-\varepsilon\sin\theta\,a(r,\theta),
  \qquad
  \dot\theta=\frac{x\dot y-y\dot x}{r^2}=-1-\frac{\varepsilon}{r}\cos\theta\,a(r,\theta).\]
Taking $\theta$ as the new independent variable,
\[  \frac{dr}{d\theta}=\frac{\dot r}{\dot\theta}
  =\frac{-\varepsilon\sin\theta\,a(r,\theta)}{-1-\frac{\varepsilon}{r}\cos\theta\,a(r,\theta)}
  =:\varphi(\varepsilon). \]
Since $\varphi(0)=0$, the quotient rule gives
\[  \varphi'(0)=\left.\frac{-\sin\theta\,a\cdot\bigl(-1-\tfrac{\varepsilon}{r}\cos\theta\,a\bigr)+\varepsilon\sin\theta\,a\cdot\bigl(-\tfrac{1}{r}\cos\theta\,a\bigr)}{\bigl(-1-\tfrac{\varepsilon}{r}\cos\theta\,a\bigr)^2}\right|_{\varepsilon=0}
  =\sin\theta\,a(r,\theta),\]
so that, by Taylor's theorem,
\begin{equation}\label{eq:drdt}
\frac{dr}{d\theta}=\varphi(\varepsilon)=\varepsilon\,F(r,\theta)+O(\varepsilon^2),
\end{equation}
where
\begin{equation}
F(r,\theta)=\sin\theta\,a(r,\theta)
=f(r\cos\theta)\cdot r\sin^2\theta
+\sgn(\cos\theta)\cdot\sin\theta\cdot g(r\cos\theta,r\sin\theta).
\end{equation}
This is in the form~\eqref{eq:standard} with period $T=2\pi$.

\begin{remark}[Sign convention: from $\theta$-stability to $t$-stability]\label{rem:sign-flip}
Equation~\eqref{eq:drdt} expresses $dr/d\theta$, not $\dot r$, and the two differ by the factor $\dot\theta=-1+O(\varepsilon)$ computed at the start of Section~\ref{sec:proof}, since the linear center $\dot x=y,\dot y=-x$ rotates clockwise, $\theta$ \emph{decreases} as $t$ increases along every orbit. By the chain rule,
\begin{equation}\label{eq:rdot_vs_dtheta}
\dot r=\frac{dr}{d\theta}\,\dot\theta
=\bigl(\varepsilon F_0(r)+O(\varepsilon^2)\bigr)\bigl(-1+O(\varepsilon)\bigr)=-\varepsilon F_0(r)+O(\varepsilon^2).
\end{equation}
Consequently, at a simple positive root $r_k$ of $F_0$, linearizing~\eqref{eq:rdot_vs_dtheta} around $r_k$ gives $\dot r\approx-\varepsilon F_0'(r_k)(r-r_k)$, so for $\varepsilon>0$ small the corresponding limit cycle is asymptotically stable in real time $t$ \emph{iff $F_0'(r_k)>0$}, and unstable iff $F_0'(r_k)<0$, the opposite sign convention from the one that governs~\eqref{eq:drdt} directly. (For $\varepsilon<0$ the roles simply swap, since $-\varepsilon F_0'(r_k)$ changes sign with $\varepsilon$.)
\end{remark}

Now, we proceed to realize the  computations of $I(r)$ and $J(r)$. The averaged function is
\begin{equation}
F_0(r)=\frac{1}{2\pi}\int_0^{2\pi}F(r,\theta)\,d\theta
=\frac{1}{2\pi}\bigl(I(r)+J(r)\bigr),
\end{equation}
where
\begin{align}
I(r)&=\int_0^{2\pi}f(r\cos\theta)\cdot r\sin^2\theta\,d\theta,\\
J(r)&=\int_0^{2\pi}\sgn(\cos\theta)\cdot\sin\theta\cdot
g(r\cos\theta,r\sin\theta)\,d\theta.
\end{align}

Write $f(x)=\sum_{i=0}^n a_i x^i$. Then
\[  I(r)=\sum_{i=0}^n a_i r^{i+1}\int_0^{2\pi}\cos^i\theta\sin^2\theta\,d\theta.\]
Using the standard Wallis-type integral formulas (see, e.g.,~\cite{DeAbreu2024}), we obtain that
\begin{equation}
\int_0^{2\pi}\cos^{2k+1}\theta\sin^2\theta\,d\theta=0,\qquad  \int_0^{2\pi}\cos^{2k}\theta\sin^2\theta\,d\theta
=\frac{\pi(2k)!}{4^k(k!)^2}\cdot\frac{1}{k+1}=:\pi\alpha_k>0,
\end{equation}
for $k=0,1,2,\ldots$, where $\alpha_k=\dfrac{(2k)!}{4^k(k!)^2(k+1)}$. Therefore  
\begin{equation}\label{eq:I_result}
I(r)=\pi\sum_{k=0}^{\floor{n/2}}\alpha_k\,a_{2k}\,r^{2k+1}.
\end{equation}
$I(r)$ is a polynomial in $r$ with $\lfloor n/2\rfloor+1$ terms, all of \emph{odd} degree.

Write $g(x,y)=\sum_{i+j\leq m}b_{ij}x^iy^j$.  Substituting $x=r\cos\theta$, $y=r\sin\theta$:
\[  J(r)=\sum_{\substack{i+j\leq m\\i,j\geq 0}}b_{ij}\,r^{i+j}\,M_{i,j},\]
where $M_{i,j}=\int_0^{2\pi}\sgn(\cos\theta)\,\sin\theta\,\cos^i\theta\,\sin^j\theta \,d\theta,\ i,j\geq 0 $. By Lemma~\ref{lem:Mij}, only the terms with both $i$ and $j$ odd contribute.  Writing $i=2p+1$, $j=2q+1$, the exponent of $r$ is $i+j=2(p+q+1)$, which is always \emph{even}. Therefore
\begin{equation}\label{eq:J_result}
J(r)=\sum_{\substack{p,q\geq 0\\2(p+q+1)\leq m}}
b_{2p+1,\,2q+1}\,M_{2p+1,\,2q+1}\,r^{2(p+q+1)}.
\end{equation}
Setting $s=p+q+1\geq 1$, the exponents of $r$ that appear in $J(r)$ are $2,4,\ldots,2\lfloor m/2\rfloor$, giving $\lfloor m/2\rfloor$ distinct even powers.

\begin{remark}\label{rem:no_r0}
The absence of the constant term $r^0$ in $J(r)$ is due to the constraint $i,j\geq 1$, which forces $i+j\geq 2$.  This is the structural reason why $H_{m,n}=h_{m,n}-1$: in the Abreu--Martins setting, the term $b_0$ of $g(x)$ contributes $r^0$ to $F_0$ via $J_{\mathrm{AM}}(r)=\sum_{k=0}^{\lfloor
m/2\rfloor}\frac{4}{2k+1}b_{2k}r^{2k}$, which starts at $r^0$; here, no such constant term exists.
\end{remark}

Combining~\eqref{eq:I_result} and~\eqref{eq:J_result}, we obtain that
\begin{equation}\label{eq:F0_full}
F_0(r)=\frac{1}{2\pi}\left(
\pi\sum_{k=0}^{\floor{n/2}}\alpha_k a_{2k}r^{2k+1}
+\sum_{\substack{p,q\geq 0\\p+q\leq\floor{m/2}-1}}
b_{2p+1,2q+1}M_{2p+1,2q+1}r^{2(p+q+1)}
\right).
\end{equation}
The odd powers of $r$ in $F_0$ are $\{r^1,r^3,\ldots,r^{2\lfloor n/2\rfloor+1}\}$ (from $I$) and the even powers are $\{r^2,r^4,\ldots,r^{2\lfloor m/2\rfloor}\}$ (from $J$).  These two sets are \emph{disjoint}, so the total number of distinct monomials in $F_0$ is
\[ \underbrace{(\floor{n/2}+1)}_{\text{odd powers from }I} +\underbrace{\floor{m/2}}_{\text{even powers from }J} =\floor{n/2}+\floor{m/2}+1. \] By Theorem~\ref{thm:descartes}, $F_0(r)$ has at most \[ (\floor{n/2}+1+\floor{m/2})-1=\floor{n/2}+\floor{m/2}=H_{m,n}\] positive real roots (the maximum number of sign changes among $\lfloor n/2\rfloor+\lfloor m/2\rfloor+1$ terms is $\lfloor n/2\rfloor+\lfloor m/2\rfloor$).\medskip

We show that the coefficients $a_{2k}$ and $b_{2p+1,2q+1}$ can be chosen so that $F_0$ has exactly $H_{m,n}$ \emph{simple} positive roots. Let $S=\{s_1<s_2<\cdots<s_t\}$ denote the set of exponents occurring in~\eqref{eq:F0_full}, that is,
\[S=\{1,3,\ldots,2\lfloor n/2\rfloor+1\}\ \cup\ \{2,4,\ldots,2\lfloor m/2\rfloor\},
\qquad t=\lfloor n/2\rfloor+\lfloor m/2\rfloor+1=H_{m,n}+1, \]
as established in the previous subsection. Fix any $H_{m,n}=t-1$ pairwise distinct positive reals $r_1<r_2<\cdots<r_{H_{m,n}}$. By Proposition~\ref{prop:achieve}, applied to the exponents $s_1,\ldots,s_t$ and the nodes $r_1,\ldots,r_{H_{m,n}}$, there exist reals $c_1,\ldots,c_t$, not all zero, such that
\begin{equation}
  P(r)=\sum_{k=1}^{t}c_k\,r^{s_k}
\end{equation}
vanishes exactly at $r_1,\ldots,r_{H_{m,n}}$, each root being simple, and $P$ has no other positive root.

It remains to realize $P$ as $2\pi F_0$ for a suitable choice of the original parameters $a_i,b_{ij}$. By~\eqref{eq:F0_full}, each odd exponent $s_k\in\{1,3,\ldots,2\lfloor n/2\rfloor+1\}$ corresponds to exactly one free coefficient of $F_0$, namely $\pi\alpha_{(s_k-1)/2}\,a_{s_k-1}$, controlled by the single parameter $a_{s_k-1}$; we set $a_{s_k-1}=c_k/(\pi\alpha_{(s_k-1)/2})$. Each even exponent $s_k\in\{2,4,\ldots,2\lfloor m/2\rfloor\}$ corresponds to a coefficient that is a sum of contributions $b_{2p+1,2q+1}M_{2p+1,2q+1}$ over all pairs $(p,q)$ with $2(p+q+1)=s_k$; we choose one such pair, set $b_{2p+1,2q+1}=c_k/M_{2p+1,2q+1}$, and set every other coefficient $b_{2p'+1,2q'+1}$ with $p'+q'=p+q$ equal to zero. With these choices, $2\pi F_0(r)=P(r)$ identically, so $F_0$ has exactly the prescribed roots $r_1,\ldots,r_{H_{m,n}}$, each simple, and no other positive root.\medskip

Finally, we proceed with the verification of the hypotheses of Theorem~\ref{thm:averaging}. 
 
\begin{proposition}\label{prop:H1}
For $|\varepsilon|$ sufficiently small, hypothesis~\ref{H1} holds for system~\eqref{eq:main_sys}.
\end{proposition}
\begin{proof}
By Lemma~\ref{lem:crossing}, $\Sigma\setminus\{0\}\subset\Sigma_c$, so every unperturbed orbit (a circle centred at the origin) crosses $\Sigma$ transversally, at a generic crossing point, for every $r>0$. Taking $D$ to be any open bounded set containing the origin and $C=D\setminus\{0\}$ gives
hypothesis~\ref{H1}.
\end{proof}
 
\begin{proposition}\label{prop:H3}
The coefficients $a_i,b_{ij}$ can be chosen so that hypothesis~\ref{H3} holds at every positive root of $F_0$, with exactly $H_{m,n}$ such roots.
\end{proposition}
 \begin{proof}
The functions $F(r,\theta)$ and $R(r,\theta,\varepsilon)$ in~\eqref{eq:drdt} are polynomials in $r$ on each half-plane $\{\cos\theta>0\}$ and $\{\cos\theta<0\}$, hence locally Lipschitz in $r$ and $2\pi$-periodic in
$\theta$; this is hypothesis~\ref{H2}.
 
By a previous construction, the coefficients $a_i,b_{ij}$ can be chosen so that $F_0$ has $r_1,\ldots,r_{H_{m,n}}$ as its only positive roots, each simple. At a simple root $r_\ell$, the Brouwer degree with respect to a sufficiently small neighbourhood $U$ of $r_\ell$ equals $d_B(F_0,U,0)=\operatorname{sign}(F_0'(r_\ell))\neq0$, since $F_0'(r_\ell)\neq0$. Hence hypothesis~\ref{H3} holds at every $a=r_\ell$, $\ell=1,\ldots,H_{m,n}$.
\end{proof}
 
By Propositions~\ref{prop:H1} and~\ref{prop:H3} and Theorem~\ref{thm:averaging}, we obtain that each simple root $r_\ell$ of $F_0$ produces, for $|\varepsilon|$ sufficiently small, a limit cycle of~\eqref{eq:main_sys} bifurcating from the periodic orbit of radius $r_\ell$ of the linear center. This completes the proof of Theorem~\ref{thm:main}.

\section{Example}

We exhibit a concrete system of the form~\eqref{eq:main_sys}, with $n=4$ and $m=4$, having exactly
\[
  H_{4,4}=\left\lfloor\frac{4}{2}\right\rfloor+\left\lfloor\frac{4}{2}\right\rfloor=4
\]
limit cycles; that is, one realizing the maximum number of limit cycles permitted by the bound of Theorem~\ref{thm:main} for this pair $(n,m)$.

Take $f(x)=a_0+a_2x^2+a_4x^4$ and $  g(x,y)=b_{11}xy+b_{13}xy^3+b_{31}x^3y,$ so that every monomial appearing in $f$ and $g$ has degree at most $m=4$, as required. By~\eqref{eq:I_result} and~\eqref{eq:J_result},
\begin{align*}
  I(r) &= \pi\left(\alpha_0a_0\,r+\alpha_1a_2\,r^3+\alpha_2a_4\,r^5\right)
        = \pi\left(a_0\,r+\frac{a_2}{4}r^3+\frac{a_4}{8}r^5\right), \\
  J(r) &= M_{1,1}b_{11}\,r^2+M_{1,3}b_{13}\,r^4+M_{3,1}b_{31}\,r^4,
\end{align*}
where $M_{1,1}=4/3$, $M_{1,3}=4/5$, and $M_{3,1}=8/15$ (recall that $M_{i,j}$ is not symmetric in $i,j$, so $M_{1,3}\neq M_{3,1}$ in general). Collecting the even powers of $r$ contributed by $J$,
\[
  [r^2]J=\frac{4}{3}b_{11}, \qquad [r^4]J=\frac{4}{5}b_{13}+\frac{8}{15}b_{31}.
\]
No term of order $r^6$ can appear, since any such term in $J$ would require $i+j=6$ for a monomial $x^iy^j$ in $g$, exceeding the degree bound $m=4$. Adding $I(r)$ and $J(r)$ then gives
\begin{equation}
  F_0(r)=\frac{1}{2\pi}\left(
    \pi a_0\,r
    +\frac{4}{3}b_{11}\,r^2
    +\frac{\pi}{4}a_2\,r^3
    +\left(\frac{4}{5}b_{13}+\frac{8}{15}b_{31}\right)r^4
    +\frac{\pi}{8}a_4\,r^5
  \right),
\end{equation}
a quintic with all five monomials of degree $1,\dots,5$ present. By Descartes' rule of signs, $F_0$ therefore has at most $4=H_{4,4}$ positive roots, the bound we aim to attain.

As in the sharp example of de Abreu and Martins~\cite{DeAbreu2024} for the horizontal-switching case, we prescribe the positive roots of $F_0$ at consecutive integers and recover the coefficients via the resulting Vandermonde-type linear system. We now choose the coefficients so that $F_0$ vanishes precisely at $r_1 = 1, r_2 = 2, r_3 = 3, r_4 = 4$. To reduce the number of free parameters to match the number of equations, set $b_{13}=0$, so that $[r^4]J=\frac{8}{15}b_{31}$, and abbreviate
\[
  A:=\pi a_0, \quad
  B:=\frac{4}{3}b_{11}, \quad
  C:=\frac{\pi}{4}a_2, \quad
  D:=\frac{8}{15}b_{31}, \quad
  E:=\frac{\pi}{8}a_4,
\]
so that $F_0(r)\propto Ar+Br^2+Cr^3+Dr^4+Er^5$. Imposing $F_0(r_\ell)=0$ for $\ell=1,2,3,4$ and dividing each equation by $r_\ell>0$ yields the homogeneous linear system
\begin{equation}
  \begin{pmatrix}
    1 & 1 & 1 & 1 & 1 \\
    1 & 2 & 4 & 8 & 16 \\
    1 & 3 & 9 & 27 & 81 \\
    1 & 4 & 16 & 64 & 256
  \end{pmatrix}
  \begin{pmatrix} A \\ B \\ C \\ D \\ E \end{pmatrix}
  = \mathbf{0}.
\end{equation}
This $4\times5$ system has a one-dimensional solution space, which we parametrize by setting $E=1$. Rather than inverting the resulting $4\times4$ Vandermonde system directly, its coefficient matrix has nodes $1,2,3,4$ and determinant $\prod_{1\le i<j\le4}(j-i)=12\neq0$, so a unique solution exists . It is quicker to note that the quintic $p(r):=Ar+Br^2+Cr^3+Dr^4+r^5$ vanishes at $r=0,1,2,3,4$ by construction, and a degree-$5$ polynomial is determined up to scalar by five roots. Hence $p(r)$ must coincide with $r(r-1)(r-2)(r-3)(r-4)=r^5-10r^4+35r^3-50r^2+24r,$ and reading off coefficients gives $A=24$, $B=-50$, $C=35$, $D=-10$, $E=1$ directly, without solving any linear system.

Undoing the substitutions above,
\begin{equation}
  a_0=\frac{A}{\pi}=\frac{24}{\pi}, \quad
  b_{11}=\frac{3B}{4}=-\frac{75}{2}, \quad
  a_2=\frac{4C}{\pi}=\frac{140}{\pi}, \quad
  b_{31}=\frac{15D}{8}=-\frac{75}{4}, \quad
  a_4=\frac{8E}{\pi}=\frac{8}{\pi},
\end{equation}
and system~\eqref{eq:main_sys} becomes
\begin{equation}\label{eq:example_system}
  \dot x=y, \qquad
  \dot y=-x-\varepsilon\left[
    \left(\frac{24}{\pi}+\frac{140}{\pi}x^2+\frac{8}{\pi}x^4\right)y
    +\sgn(x)\left(-\frac{75}{2}xy-\frac{75}{4}x^3y\right)
  \right].
\end{equation}
By construction, the averaged function of~\eqref{eq:example_system} satisfies $F_0(1)=F_0(2)=F_0(3)=F_0(4)=0$ with $F_0(r)\propto r(r-1)(r-2)(r-3)(r-4)$ (Figure~\ref{fig:F0}); since these four roots are simple, Theorem~\ref{thm:averaging} guarantees exactly $4$ limit cycles bifurcating from the origin for $|\varepsilon|$ sufficiently small.

\begin{remark}[Stability of the limit cycles]
Since $F_0(r)\propto r(r-1)(r-2)(r-3)(r-4)$, the product rule gives, at each root $r_k\in\{1,2,3,4\}$,
\[
  F_0'(r_k)\propto r_k\prod_{\ell\neq k}(r_k-r_\ell),
\]
where the factor $r_k>0$ does not affect the sign. Evaluating,
\begin{align*}
  F_0'(1)&\propto 1\cdot(-1)(-2)(-3)=-6<0, &&\text{unstable}, \\
  F_0'(2)&\propto 2\cdot(1)(-1)(-2)=+4>0, &&\text{stable}, \\
  F_0'(3)&\propto 3\cdot(2)(1)(-1)=-6<0, &&\text{unstable}, \\
  F_0'(4)&\propto 4\cdot(3)(2)(1)=+24>0, &&\text{stable}.
\end{align*}
By Remark~\ref{rem:sign-flip}, the four cycles therefore alternate unstable–stable–unstable–stable see Figure~\ref{fig:phase}). Direct numerical integration of~\eqref{eq:example_system}, not merely of the averaged equation, confirms this alternating pattern for initial radii on both sides of each $r_k$.
\end{remark}

\begin{figure}[H]
  \centering
  \includegraphics[width=0.80\linewidth]{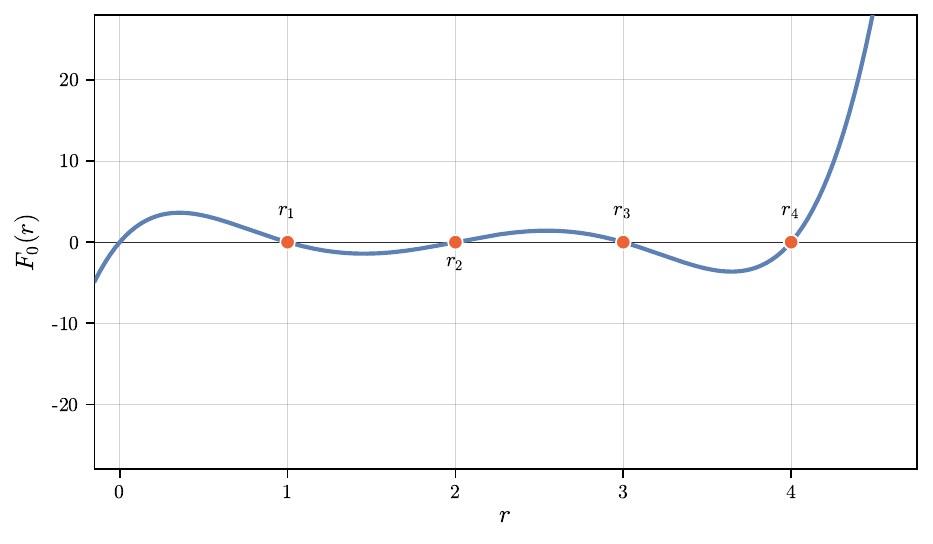}
  \caption{Graph of $F_0(r)$ for system~\eqref{eq:example_system}. The four positive roots $r_1=1,r_2=2,r_3=3,r_4=4$ correspond to four limit cycles; the alternating sign of $F_0'$ at consecutive roots reflects their alternating stability.}
  \label{fig:F0}
\end{figure}

\begin{figure}[H]
  \centering
  \includegraphics[width=0.60\linewidth]{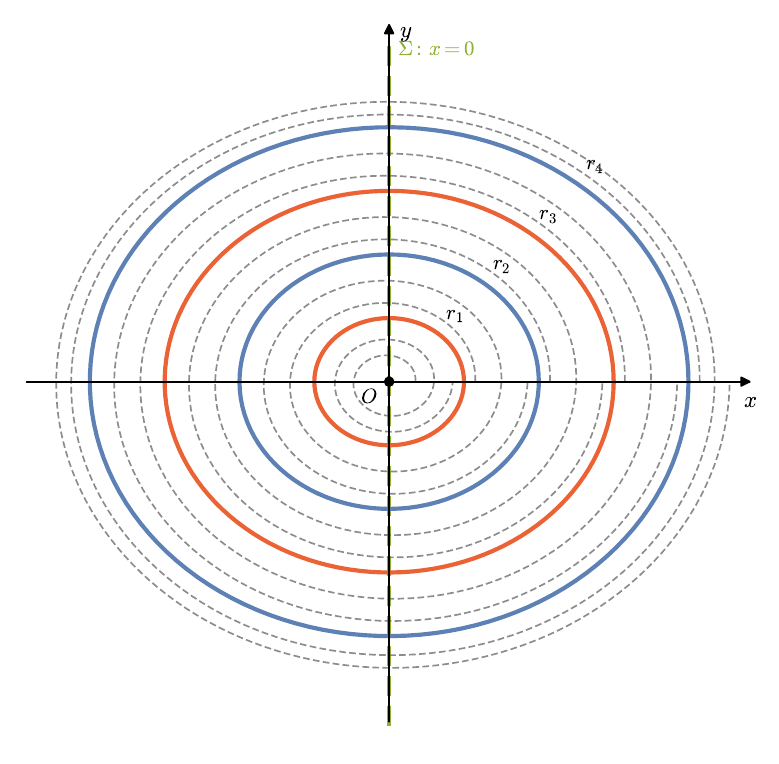}
  \caption{Schematic phase portrait of~\eqref{eq:example_system}. Solid curves are the four limit cycles, alternating stable (blue) and unstable (orange); dashed curves are representative trajectories spiralling toward or away from them. The dashed line $\Sigma=\{x=0\}$ is the switching manifold.}
  \label{fig:phase}
\end{figure}

\noindent\textbf{Funding.} No funding was received for this study.\medskip

\noindent\textbf{Author contributions.} CFA proved the main results and typed the article. ASO proved the main results and typed the article. All authors reviewed the manuscript.\medskip

\noindent\textbf{Data availability.} No datasets were generated or analyzed during the current study.

\section*{Declarations}

\noindent\textbf{Conflict of interest.} The authors declare no competing interests relevant to the content of this article.\medskip

\noindent\textbf{Use of AI tools.} During the preparation of this manuscript, the authors used an AI assistant (Claude, Anthropic) to improve the writing and to verify mathematical computations. All mathematical content, proofs, and conclusions are the sole responsibility of the authors.

\section*{Acknowledgments}
The authors thank the referees for their useful suggestions and corrections, which helped to improving this work.

\bibliography{Bibliography}
\bibliographystyle{acm}

\end{document}